\documentclass[11pt]{amsart}
\pdfoutput=1
\usepackage{amssymb}
\usepackage{amsmath}
\usepackage{amsfonts}
\usepackage[usenames]{color}
\usepackage{graphicx}
\usepackage{array}
\usepackage{psfrag}
\usepackage{color}
\usepackage{ulem}
\usepackage{fancyhdr}

\usepackage{caption}

\usepackage{import}
\usepackage{xifthen}
\usepackage{pdfpages}
\usepackage{transparent}

\makeatletter
 
 \@addtoreset{equation}{section}
\makeatother

\newtheorem{theorem}{Theorem}
\newtheorem{lemma}[theorem]{Lemma}
\newtheorem{proposition}[theorem]{Proposition}

\theoremstyle{definition}

\newtheorem{observation}[theorem]{Observation}

\newtheorem{corollary}[theorem]{Corollary}

\theoremstyle{remark}
\newtheorem{remark}[theorem]{Remark}

\numberwithin{equation}{section}

\newcommand{\R}{\mathbb{R}}
\newcommand{\N}{\mathbb{N}}

\newcommand{\dfn}[1]{\textit{#1}}

\begin{document}

\title{{Intrinsically knotted graphs and connected domination}}

\date{\today}
\author{
 Gregory Li, Andrei Pavelescu, and Elena Pavelescu
}

\address{\textit{gregoryli@college.harvard.edu}, Harvard University,
Cambridge, MA 02138}
\address{\textit{andreipavelescu@southalabama.edu}, Department of Mathematics and Statistics, University of South Alabama, Mobile, AL  36688}
\address{\textit{elenapavelescu@southalabama.edu}, Department of Mathematics and Statistics, University of South Alabama, Mobile, AL  36688}

\maketitle
\rhead{Intrinsically knotted graphs and connected domination}

\begin{abstract}

We classify all the maximal linklessly embeddable graphs of order 12 and show that their complements are all intrinsically knotted. We derive results about the connected domination numbers of a graph and its complement. We provide an answer to an open question about the minimal order of a 3-non-compliant graph. We prove that the complements of knotlessly embeddable graphs of order at least 15 are all intrinsically knotted. We provide results on general $k$-non-compliant graphs and leave a set of open questions for further exploration of the subject.
\end{abstract}


\section{Introduction}

Sixty years ago, Battle, Harary, and Kodama \cite{BHK} proved that the complement of a planar graph of order nine is not planar. This result was independently proved by Tutte \cite{Tutte}. We express this fact by saying that the complete graph $K_9$ is not \textit{bi-planar}. In general, for a graph property $\mathcal{P}$, we say that $K_n$ is \textit{bi-$\mathcal{P}$}, if there exists a graph $G$ of order $n$ such that both $G$ and its complement, $\overline{G}$, have the property $\mathcal{P}$. 

Since then, many similar results were derived, some of them generalizing the original one. Ichihara and Mattman \cite{IM} proved that for all nonnegative integers $t$, the complete graph $K_{2t+9}$ is not bi-t-apex. A graph is called \textit{t-apex} if there exist a choice of $t$ of its vertices which deleted yield a planar graph. In private communications, the last two authors argued that this result is sharp, as they showed that $K_{2t+8}$ is bi-t-apex for all $t\ge 0$.  Odeneal, Naimi, and the last two authors \cite{NOPP} proved that $K_n$ is not bi-linklessly embeddable (bi-nIL), for all $n\ge 11$. A graph is \dfn{intrinsically linked} (\dfn{IL}) if every embedding of it in $\R^3$ (or $S^3$) contains a nonsplittable 2-component link.
A graph is \dfn{linklessly embeddable} if it is not intrinsically linked (\dfn{nIL}). The $n\ge 11$ bound is also sharp. 

In this article, we investigate  bi-knotless embeddability. A graph is called \dfn{intrinsically knotted} if every embedding of it in $S^3$ contains a nontrivial knot.
We abbreviate intrinsically knotted as IK, and not intrinsically knotted as nIK. We are trying to provide an answer to the following question: what is the smallest integer $n$, such that $K_n$ is not bi-nIK? We employ some of the techniques used by the last two authors in \cite{PP}, where they proved $K_{13}$ is not bi-nIL. In the process, we find all the 6503 maximal nIL graphs of order 12. We focus on finding a large degree vertex in a minor of either the graph or its complement. This, in turn, leads to the study of the connected domination numbers for both $G$ and $\overline{G}$. We then prove:\\

{\bf Theorem} The complete graph $K_n$ is not bi-nIK, for all $n\ge 15$.\\

 All the graphs considered in this article are simple: non-oriented, without loops or multiple edges. For a given graph $G$, we let $V(G)$ denote its vertex set, and $E(G)$ denote its edge set. Given a graph $G$, the \textit{complement} graph $\overline{G}$ has the same vertex set as $G$ and $E(\overline{G}) = \{ (u, v) \:|\: u, v \in V(G) \: \text{and} \: (u, v) \not \in E(G) \}.$
The maximum degree among all vertices of $G$ is $\Delta (G)$, and the minimum degree is $\delta (G)$.
Let $\delta^*(G):= \min\{\delta(G), \delta(\overline{G})\}$. 
For a vertex $u \in V(G)$, the set $N_G(u)$ is the set of all vertices of $G$ that are adjacent to $u$, and $N_G[u]= N_G(u)\sqcup \{u\}$. For a given graph $G$, we denote the number of edges (\textit{size}) of $G$ by $|G|$.

A set $S \subseteq V(G)$ is a \textit{dominating set} of $G$  if every vertex in $V \backslash S$ is adjacent to at least one vertex in $S$.
The \textit{domination number} $\gamma(G)$ of $G$ is the minimum cardinality of dominating sets $S$ of $G$.
The \textit{connected domination number} $\gamma_c(G)$ of $G$ is the minimum cardinality of dominating sets $S$ of $G$ which induce a \textit{connected} subgraph $G[S]$ of $G$.

Some of the necessary results on connected domination were readily available by the work of Karami, Sheikholeslami, Khodkar, and West \cite{KSKW}. They proved a Nordhaus-Gaddum type relation for the connected domination number.  

\begin{theorem}[Karami et al. \cite{KSKW}]
\label{th-karami}
If $G$ and $\overline{G}$ are both connected graphs on $n$ vertices with $\gamma_c(G), \gamma_c(\overline{G}) \geq 4$, then $$\gamma_c(G) + \gamma_c(\overline{G}) \leq \delta^*(G) + 2.$$
\end{theorem}

They showed that equality can only hold when $ \delta^*(G)=6$, and they provided an example of a simple graph $G$ with $\gamma_c(G) =\gamma_c(\overline{G})=4$, and $ \delta^*(G)=6$. However, this example has nearly 15,000 vertices. We prove that there is a minimal example of order 13.

A \textit{minor} of a graph $G$ is any graph that can be obtained from $G$ by a sequence of vertex deletions, edge deletions, and edge contractions.
An edge contraction is performed by identifying the endpoints of an edge in $E(G)$ and then deleting any loops and double edges thus created.

For $k\ge 1$, we say that $G$ is \textit{k-compliant} when $G$ or $\overline{G}$ contains a minor $H$ with $\Delta(H) \geq |G|-k$. Otherwise, $G$ is  \textit{k-non-compliant}.
It follows that for  $k\ge 1$, if $G $ is $k$-compliant then $G$ is $(k+1)$-compliant. Notice that a graph $G$ is 1-compliant if and only if the graph or its complement have an isolated vertex/ a dominating vertex (cone). A graph is 2-compliant if and only if, in either the graph or its complement, one edge contraction yields a dominant vertex. We also note that a disconnected graph is 2-compliant. 
We show in Theorem \ref{theo:compliance} that  $G$ is $k$-compliant if and only if $\min\{\gamma_c(G), \gamma_c(\overline{G})\} \leq k$. In Section 5, we provide results about general $k$-non-compliant graphs. In this setting, the question posed by Karami et al. \cite{KSKW} is equivalent to finding the minimum order of a 3-noncompliant graph. We answer this with Theorem 6.\\

{\bf Theorem} Up to isomorphism, the Paley 13-graph $QR_{13}$ is the only 3-non-compliant graph of order 13.\\

We also find all 3-non-compliant graphs of order 14.\\

{\bf Theorem} Up to isomorphism, there are only two 3-non-compliant graphs of order 14: the graph obtained by adding a new vertex to $QR_{13}$ connected to all the vertices of the open neighborhood of an existing vertex, and the graph obtained by adding a new vertex to $QR_{13}$ connected to all the vertices of the closed neighborhood of an existing vertex.


\section{Three-Non-Compliant Graphs}

The following result follows directly from the definition of a 3-non-compliant graph. 

\begin{lemma}
\label{3set}
Let $G$ be a graph that is $3$-non-compliant, and let $S \subseteq V(G)$ be such that $|S| = 3$. Then
\begin{itemize}
    \item if $G[S]$ is connected, there exists $v \in V \setminus S$ such that $(u,v) \not \in E(G)$ for all $u \in S$
    \item if $\overline{G}[S]$ is connected, there exists $v \in V \setminus S$ such that $(u,v) \in E(G)$ for all $u \in S$ 
\end{itemize}
\end{lemma}

\begin{lemma}
\label{j3}
For $G$ a  3-non-compliant graph, 
$$|N_G(u) \cap N_G(v)| \geq  3,$$
for any $(u,v) \not \in E(G). $
\end{lemma}

\begin{proof}
Let us assume on the contrary $G$ is 3-non-compliant and $|N(u) \cap N(v)| \leq 2$ for $(u,v) \not \in E(G)$.
 If there exists a vertex $w \in N_{\overline{G}}(u) \cup N_{\overline{G}}(v)$ such that $(w, v_i) \in E(\overline{G})$ for each $v_i \in N_G(u) \cap N_G(v)$, then we see $\gamma_c(\overline{G}) \leq 3$ by taking $S = \{u, v, w\}$.

Otherwise, for each vertex  $w \in N_{\overline{G}}(u) \cup N_{\overline{G}}(v)$, we have $(w, v_i) \in E(G)$ for some $v_i \in N_G(u) \cap N_G(v)$.
 By taking $S = (N(u) \cap N(v)) \sqcup \{u\}$ or $(N(u) \cap N(v)) \sqcup \{v\}$, we see $\gamma_c(G) \leq 3$. 

In either case, for $G$ 3-non-compliant, $|N(u) \cap N(v)| \geq  3$ for $(u,v) \not \in E(G)$, a contradiction.
\end{proof}

\begin{lemma}
\label{i2}
For $G$ a 3-non-compliant graph, $$|N_G(u) \cap N_G(v)| \geq 2,$$
for any $(u,v) \in E(G)$.
\end{lemma}

\begin{proof}
    Let us assume on the contrary $G$ is 3-non-compliant and $|N(u) \cap N(v)| \leq 1$ for $(u,v) \in E(G)$. If  $|N(u) \cap N(v)|=0$, since $G$ is 3-non-compliant, it must be that there exists $w\in V(G)\setminus \{u,v\}$ such that $(v,w),(u,w)\notin E(G)$. By Lemma \ref{3set}, it follows that $u$, $v$, and $w$ must share a neighbor in $V(G)$, thus $|N(u) \cap N(v)|>0$. Assume $|N(u) \cap N(v)|=1$, that is there is a vertex $w_0 \in V(G)$ such that $|N(u) \cap N(v)|=\{w_0\}$. In $\overline{G}$, by Lemma \ref{3set}, the set $\{u,v,w_0\}$ must share a common neighbor, say $x$. Then $\{u,v,x\}$ forms a connected 3-dominating set in $\overline{G}$, a contradiction.
\end{proof}

\begin{corollary}
\label{deg6}
If $G$ is 3-non-compliant, then $|G|  \geq 13.$
\end{corollary}
\begin{proof}
Let $G$ be 3-non-compliant graph. 
Then $\gamma_c(G), \gamma_c(\overline{G}) \geq 4$, and both $G$ and $\overline{G}$ are connected.
Theorem \ref{th-karami} implies $8 \le  \delta^*(G)+2,$ that is $ \delta^*(G)\ge 6$.
The smallest possible order for $G$ with $ \delta^*(G)\ge 6$ is $n=13$.
\end{proof}


If $G$ is a 3-non-compliant graph of order 13, Corollary \ref{deg6} shows $G$ must be 6-regular. There are 367,860 non-isomorphic 6-regular graphs of order 13. However, only one of them is 3-non-compliant. In \cite{PP}, the last two authors gave another proof for the following theorem. We present a shorter proof which uses Lemmas \ref{j3} and \ref{i2}. 

\begin{theorem} Up to isomorphism, the Paley 13-graph $QR_{13}$ is the only 3-non-compliant graph of order 13.
\label{QR13}
\end{theorem}
\begin{proof}
By Corollary \ref{deg6}, if $G$ is a 3-non-compliant graph of order 13, it must be 6-regular. 
Let $v\in V(G)$. Let  $N:=N_G(v)$, let $H=G\setminus V(N)\setminus \{v\}$ be the subgraph of $G$ induced by the non-neighbors of $v$, and let $L$ denote the bipartite subgraph of $G$ such that $V(L)=V(G)$, and $E(L)$ consists of those edges of $G$ with one endpoint in $N$ and the other in $H$.  
By Lemma \ref{j3}, each vertex of $H$ must be adjacent to at least 3 vertices of $N$, thus $|L|\ge 6\cdot 3=18$. By Lemma \ref{i2}, $\delta(N)\ge 2$, thus $|N|\ge 6$. Adding the degrees in $G$ of the vertices of $N$,
we get
\[36=\sum_{v\in N}deg_G(v)=6+2|N|+|L|\ge 6+12+18=36,\] which shows that $N$ is 2-regular, and $L$ is 3-regular.
Since the same must hold for every vertex of $G$, lest $G$ be 3-compliant, it follows that $G$ is a strongly regular $(13,6,2,3)$ graph, thus $G$ is isomorphic to the Paley 13-graph.\\\\
On the other hand, if $G$ is isomorphic to the Paley 13-graph, up to isomorphism, there are exactly 13 non-isomorphic minors obtained by contracting two edges of $G$, and each of them has maximum degree at most 9. 
Since the Paley 13-graph is self-complementary, it follows that $G$ is 3-non-compliant.  
\end{proof}

\begin{proposition}
\label{twin}
Let $u$ and $v$ be twin vertices of a graph $G$ ($N_G(u) =N_G(v)$). For $k\ge 2$, $G$ is $k$-non-compliant if and only if $G':=G-v$ is $k$-non-compliant.

\end{proposition}
\begin{proof} Note that if $u$ and $v$ are twin vertices of $G$, then they are also twin vertices of $\overline{G}$. So, without loss of generality, we may assume $(u,v)\in E(G)$.

Assume that $G$ is $k$-non-compliant and that $G'$ is $k$-compliant. If $\gamma_c(G')\le k$, let $S=\{u_1,...,u_k\}$ be a connected dominating set in $G'$. Since $S$ is connected, there exist $1\le i\le k$ such that $(u,u_i)\in E(G')$. Since $u$ and $v$ were twin vertices in $G$, $(v,u_i)\in E(G)$, so $S$ is a connected dominating set of $G$, a contradiction. 

If $\gamma_c(\overline{G'})\le k$, let $S=\{u_1,...,u_k\}$ is a connected dominating set of $\overline{G'}$. Since $S$ is connected, there is $1\le i\le k$ such that $(u,u_i)\in E(\overline{G'})$. As $u$ and $v$
 are twins, $(v,u_i)\in E(\overline{G})$, so $S$ dominates $\overline{G}$.\\
 
 Assume that $G$ is $k$-compliant and that $G'$ is $k$-non-compliant. If $\gamma_c(G)\le k$, let $S=\{u_1,...,u_k\}$ be a connected dominating set in $G$. As $G'$ is $k$-non-compliant, $v\in S$. Since $u$ and $v$ have the same neighbors in $G$, $(S-v)\cup \{u\}$ is a connected dominating set of $G'$ of size $k$.
A contradiction, since $G'$ is $k$-non-compliant.

 If $\gamma_c(\overline{G})\le k$, let $S=\{u_1,...,u_k\}$ be a connected dominating set in $\overline{G}$. If $v\in S$, then $(S-v)\cup \{u\}$ is a connected dominating set of $\overline{G'}$. If $v\notin S$, then $S$ is a connected dominating set of $\overline{G'}$. In either case, $G'$ must be $k$-compliant, a contradiction.
\end{proof}

Note that deleting a vertex from the Paley 13 graph yields a subgraph of order 12 which, by Corollary \ref{deg6}, it must be 3-compliant, but 2-non-compliant. 
We pose the following question:\\
 
\noindent{\textbf{Question 1}: Under what condition does adding a vertex to a $k$-compliant graph yields a $k$-non-compliant graph?\\

Proposition \ref{twin} allows one to build 3-non-compliant graphs of order at least 13 by repeatedly adding twin vertices. For order 14, one obtains two non-isomorphic structures: adding a new vertex $v$ to the open/closed neighborhood of a vertex in the $QR_{13}$ graph. Theorem \ref{order14} shows that these two are the only 3-noncompliant graphs of order 14. \\

Lemmas \ref{j3} and \ref{i2} imply that if $G$ is 3-non-compliant, then every pair of adjacent vertices have at least two common neighbors and every pair of non-adjacent vertices have at least three common neighbors.
For $n=14$, the possible degrees for a vertex of a 3-non-compliant graph are 6 and 7. 
This implies the number of edges for a 3-non-compliant graph of order 14 is $42\le m\le 49$.
To find all 3-non-compliant graphs it suffices to find those graphs with $42\le m\le 45$, then check their complements. More restrictions can be derived by inspecting the possible configurations of the neighborhoods of vertices of degree 6.


\begin{lemma}
    Let $G$ be a graph with $14$ vertices and $u\in V(G)$ with $\deg_G(u)=6$. If $G$ is 3-non-compliant, then
    $$\sum_{v\in N_G(u)}\deg_G{v}\ge 39.$$
    \label{l39}
\end{lemma}
\begin{proof}
Let $N=N_G(u)$, $H$ be the subgraph induced by $V(G)\setminus (V(N)\cup \{v\})$, and $L$ be the subgraph whose edges have one endpoint on $N$ and one endpoint in $H$. 
By Lemma \ref{j3}, each vertex of $H$ shares at least three neighbors with $u$, thus $|L|\ge 21.$
By Lemma \ref{i2}, $u$ shares at least two neighbors with each of of its neighbors, thus 
$\deg_N({v})\ge 2,$ for each $v\in N$ and $|N|\ge 6$.
We have $$\sum_{v\in N_G(u)}\deg_G(v) = 6 + 2|N| +|L|\ge 6 + 12 + 21 = 39. $$
\end{proof}

The next two lemmas partially strengthen the results of Lemmas \ref{j3} and \ref{i2} for graphs of order 14.

\begin{lemma} Let $G$ be a 3-non-compliant graph of order 14. Let $u,v \in V(G)$, such that $deg_G(u)=deg_G(v)=7$, and $(u,v)\notin E(G)$. Then $|N_G(u)\cap N_G(v)|\ge 4$.
\label{j4}
\end{lemma}
\begin{proof}
By Lemma \ref{i2}, $u$ and $v$ share at least two neighbors in $\overline{G}$. As $u$ and $v$ are not adjacent in $G$, we have
\[10\ge |N_G(u)\cup N_G(v)|=|N_G(u)|+|N_G(v)|-|N_G(u)\cap N_G(v)|=7+7-|N_G(u)\cap N_G(v)|\]
\[\Rightarrow |N_G(u)\cap N_G(v)|\ge 14-10=4.\]
\end{proof}

\begin{lemma} Let $G$ be a 3-non-compliant graph of order 14. Let $u,v \in V(G)$, such that $deg_G(u)=deg_G(v)=7$, and $(u,v)\in E(G)$. Then $|N_G(u)\cap N_G(v)|\ge 3$.

\label{i3}
\end{lemma}
\begin{proof}
By Lemma \ref{j3}, $u$ and $v$ share at least three neighbors in $\overline{G}$. As $u$ and $v$ are adjacent in $G$, we have
\[11\ge |N_G(u)\cup N_G(v)|=|N_G(u)|+|N_G(v)|-|N_G(u)\cap N_G(v)|=7+7-|N_G(u)\cap N_G(v)|\]
\[\Rightarrow |N_G(u)\cap N_G(v)|\ge 14-11=3.\]
\end{proof}

\begin{lemma} Graphs $G$ of order 14 with 42, 43, or 44 edges are 3-compliant. 
\label{lem:edges14}
\end{lemma}
\begin{proof}
For size $42, 43$, consider a vertex $u$ with $\deg(u)=6$. Since there are at most two vertices of degree 7 in $G$, the sum of degrees for the neighbors of $u$ is strictly less than 39. By Lemma \ref{l39}, the graph $G$ is 3-compliant.
\\
For size 44, the graph $G$ has four vertices of degree 7 and ten vertices of degree 6. 
Using the pigeonhole principle, we see that there exists a vertex of degree 6 which neighbors at most two vertices of degree 7.
By Lemma \ref{l39}, the graph $G$ is 3-compliant.
\end{proof}

\begin{lemma}  Let $G$ be a graph with $14$ vertices, size 45, and $u\in V(G)$ with $\deg_G(u)=7$. If $G$ is 3-non-compliant, then
    $$\sum_{v\in N(u)}\deg_G{v}\ge 44.$$
\label{l44}
\end{lemma}
\begin{proof}
Let $N=N_G(u)$ and let $H=N_{\overline{G}}(u).$ By Lemma \ref{l39}, 
\[\sum_{v\in H}\deg_{\overline{G}}{v} \ge 39 \Rightarrow \sum_{v\in N_{\overline{G}}[u]}\deg_{\overline{G}}{v} \ge 45. \] By the handshaking lemma, 
\[\sum_{v\in N}deg_{\overline{G}}(v)\le 2|\overline{G}|-45=92-45=47.\] Taking complements,
\[\sum_{v\in N}deg_G(v)\ge 7\cdot 13-47=91-47=44.\]

\end{proof}

\begin{theorem} Up to isomorphism, there are only two 3-non-compliant graphs of order 14: the graph obtained by adding a new vertex to $QR_{13}$ connected to all the vertices of the open neighborhood of an existing vertex, and the graph obtained by adding a new vertex to $QR_{13}$ connected to all the vertices of the closed neighborhood of an existing vertex.
\label{order14}
\end{theorem}
\begin{proof} Since a graph is 3-non-compliant if an only if its complement is 3-non-compliant, by Lemma \ref{lem:edges14} it suffices to consider simple graphs of order 14 and size 45.
Let's assume $|G|=45$. Since $\delta^\ast(G)=6$, it follows that the degree sequence of $G$ is $(6,6,6,6,6,6,6,6,7,7,7,7,7,7)$. Let $A$ denote the subgraph of $G$ induced by the vertices of degree 7, let $B$ denote the subgraph of $G$ induced by the vertices of degree 6, and let $P$ denote the bipartite graph obtained by deleting all the edges of $A$ and all the edges of $B$ from $G$. By Lemma \ref{l39}, as every vertex of degree 6 must connect to at least three vertices of degree 7, it follows that $\Delta(B)\le 3$, and $|B|\le 12$. By Lemma \ref{l44}, it follows that $\delta(A)\ge 2$, and thus $|A|\ge 6$. 

\[42=\sum_{v\in A}deg_G(v)=2|A|+|P|\ge 12+|P|\Rightarrow |P|\le 30.\]
\[48=\sum_{v\in B}deg_G(v)=2|B|+|P|\le 24+|P|\Rightarrow |P|\ge 24.\] Either relation shows that $|P|$ is even, and thus $|P|\in \{24, 26, 28, 30\}$.\\

If $|P|=24$, then $|A|=9$, $2\le \delta(A)\le 3$, $|B|=12$, and every vertex of degree 6 is adjacent to exactly three vertices of degree 7. There are 15 non-isomorphic structures possible for $A$, as seen in Figure \ref{p24}.

\begin{figure}[htpb!]
\begin{center}
\begin{picture}(370, 170)
\put(0,0){\includegraphics[width=5in]{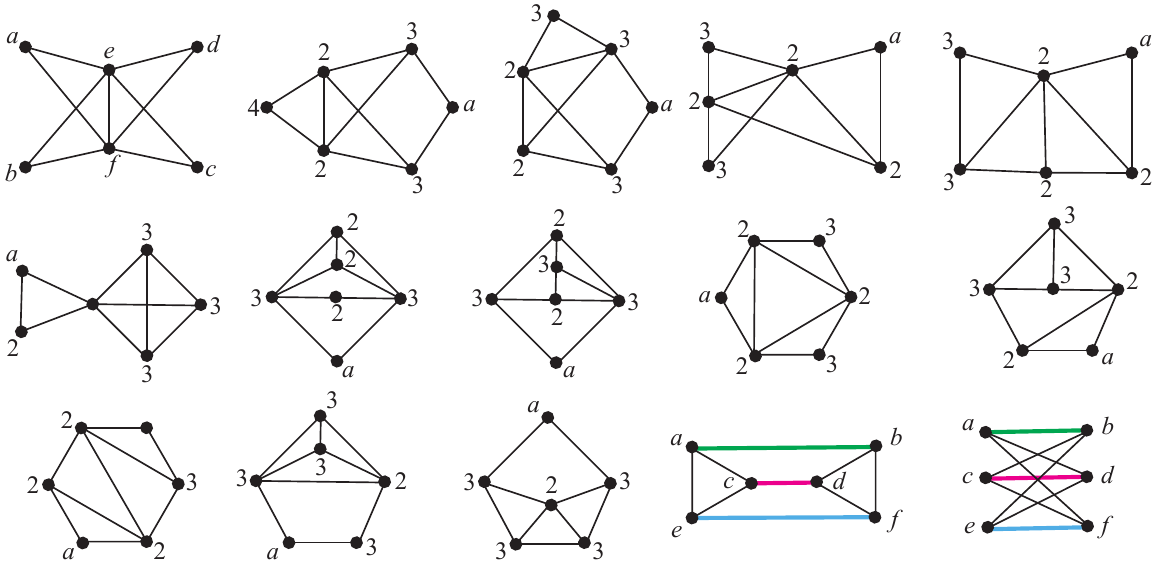}}
\end{picture}
\caption{\small The 15 graphs of order 6, size 9, and minimum degree at least 2.} 
\label{p24}
\end{center}
\end{figure}

 If $\delta(A)=3$, then $A$ is 3-regular and is therefore isomorphic to either the complete bipartite graph $K_{3,3}$, or the triangular prism (the last two graphs in Figure \ref{p24}). In either case, the edges $(a,b)$, $(c,d)$, and $(e,f)$ are non-triangular in $A$, with both their endpoints vertices of degree 7. By Lemma \ref{i3}, the endpoints of each of these edges must share at least three neighbors among the vertices of $B$. As $B$ has order 8, the pigeonhole principle implies there is a vertex of degree 6 which is adjacent to at least 4 vertices of $A$, a contradiction. 

If $\delta(A)=2$, then is A is among the first 13 graphs in Figure \ref{p24}. Any such vertex must be adjacent to 5 vertices of degree 6. Since every vertex of degree 6 neighbors exactly 3 vertices of degree 7, any such vertex is part of at most $5 \cdot 2=10$ pairs of vertices of $A$ which share a neighbor in $B$. On the other hand, without exception, there is always a vertex of degree 2, labeled $a$ for all 13 graphs in Figure \ref{p24}, for which the computed sum of the number of such pairs surpasses 10. We provide here the computation for the first graph as an example, and we leave the rest to the reader. Note that in Figures \ref{p24}, \ref{p26}, and \ref{p28}, the numerical values represent lower bounds for the common $B$-neighbors those vertices have with $a$.

By Lemma \ref{i3}, $a$ and $e$ have at least three common neighbors. Since they only share one neighbor in $A$, they must have at least two common neighbors among the vertices of $B$. The same argument works for $a$ and $f$.  By Lemma \ref{j4}, $a$ and $b$ share at least 4 neighbors, with at least two of them among the vertices of $B$. However, if $a$ and $b$ have exactly two common neighbors among the vertices of $B$, then $\{a,e,b\}$ is a connected dominating set of $G$, so $G$ is 3-compliant. This argument shows that $a$ and $b$ share at least three neighbors among the vertices of $B$. The same must be true for each of the pairs $\{a,c\}$ and $\{a,d\}$. It follows that $a$ must be in at least $2+2+3+3+3=13 >10$ pairs of vertices of $A$ with a shared vertex in $B$, a contradiction. 

So, if $|P|=24$, $G$ must be 3-compliant.\\

If $|P|=26$, then $|A|=8$, $\delta(A)=2$, $|B|=11$, and every vertex of degree 6 is adjacent to at least three vertices of degree 7. There are 11 non-isomorphic structures possible for $A$, depicted in Figure \ref{p26}.

\begin{figure}[htpb!]
\begin{center}
\begin{picture}(360, 170)
\put(0,0){\includegraphics[width=4.8in]{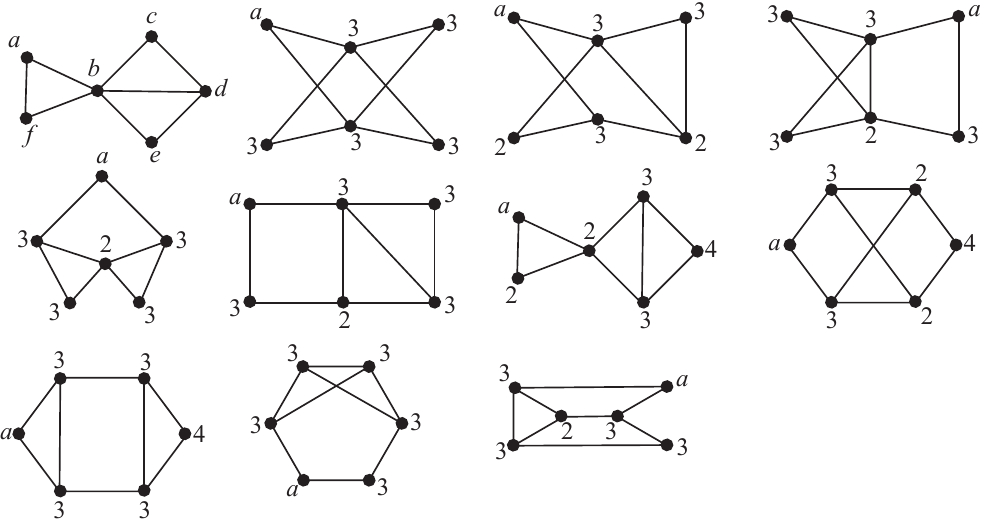}}
\end{picture}
\caption{\small The 11 graphs of order 6, size 8, and minimum degree at least 2.} 
\label{p26}
\end{center}
\end{figure}

Every vertex $v$ of $A$ with $deg_A(v)=q\ge 2$ is adjacent to $7-q$ vertices of $B$. Since every vertex of $B$ is adjacent to at least three vertices of $A$, it follows that $v$ can be in  at most $2\cdot (7-q)+26-24=16-2q$ pairs of vertices of $A$ which share a neighbor in $B$. In particular, a vertex of degree 2 of $A$ is in at most 12 pairs of vertices of $A$ which share a neighbor in $B$. However, for each of the graphs in Figure \ref{p26}, by Lemmas \ref{j4} and \ref{i3}, there exists a vertex of degree 2, labeled $a$, for which the number of this kind of pairs surpases 12. We provide the computations for the first graph in the list, and leave the rest to the reader. We note that for the second graph in Figure \ref{p26}, the argument is identical to the case covered when $|P|=24$.
By Lemma \ref{i3}, $a$ and $b$ have at least three common neighbors. Since they share a neighbor in $A$, they must share at least two neighbors in $B$. The same argument holds for the pair $a$ and $f$. Each of the pairs $\{a,c\}$, $\{a,d\}$, and $\{a,e\}$ must share at least 4 neighbors in $G$, by Lemma \ref{j4}. However, each of them only has one common neighbor in $A$, giving them each at least three common neighbors in $B$. The number of pairs of vertices of $A$ containing $a$ with a shared common neighbor in $B$ is at least $2+2+3+3+3=13>12$.

So, if $|P|=26$, $G$ must be 3-compliant.\\

If $|P|=28$, then $|A|=7$, $\delta(A)=2$, $|B|=10$, and every vertex of degree 6 is adjacent to at least three vertices of degree 7. There are 5 non-isomorphic structures possible for $A$, depicted in Figure \ref{p28}.

\begin{figure}[htpb!]
\begin{center}
\begin{picture}(300, 120)
\put(0,0){\includegraphics[width=4in]{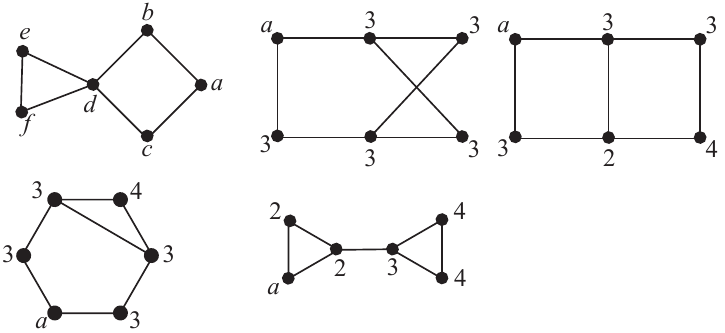}}
\end{picture}
\caption{\small The 5 graphs of order 6, size 7, and minimum degree at least 2.} 
\label{p28}
\end{center}
\end{figure}

Every vertex $v$ of $A$ with $deg_A(v)=q\ge 2$ is adjacent to $7-q$ vertices of $B$. Since every vertex of $B$ is adjacent to at least three vertices of $A$, it follows that $v$ can be in  at most $2\cdot (7-q)+28-24=18-2q$ pairs of vertices of $A$ which share a neighbor in $B$. In particular, every graph in Figure \ref{p28} has a vertex of degree 2 , labeled $a$, which is allowed to be in at most 14 pairs of vertices of $A$ which share a neighbor in $B$. As before, for each of the graphs, a direct computation will show that the number of actual pairs must surpass this threshold in order for Lemmas \ref{j4} and \ref{i3} to hold. We provide the computations for the first graph in the list, and leave the rest to the reader. 

By Lemma \ref{i3}, each pair $\{a,b\}$ and $\{a,c\}$ must have at least 3 neighbors in $B$, since neither shares neighbors in $A$. By Lemma \ref{j4}, $a$ and $d$ must have at least 4 common neighbors, at least 2 of which must be in $B$. By Lemma \ref{j4}, $\{a,e\}$ and $\{a,f\}$ must have at least 4 neighbors in $B$, since neither shares neighbors in $A$. The total number of pairs for $a$ is at least $3+3+2+4+4=16>14$, thus a contradiction.

So, if $|P|=28$, $G$ must be 3-compliant.\\

If $|P|=30$, then $|A|=6$, A is 2-regular, $|B|=9$, and every vertex of degree 6 is adjacent to at least three vertices of degree 7. There are only two structure possible for $A$, $K_3\sqcup K_3$ and $C_6$, as in Figure \ref{p30}.
\begin{figure}[htpb!]
\begin{center}
\begin{picture}(210, 45)
\put(0,0){\includegraphics[width=3in]{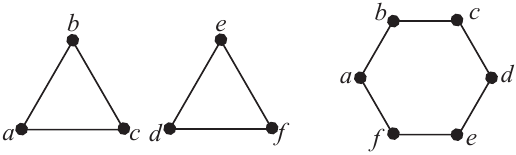}}
\end{picture}
\caption{\small The 2-regular graphs of order 6.} 
\label{p30}
\end{center}
\end{figure}

Since $A$ is two regular, each vertex of $A$ is adjacent to exactly 5 vertices of $B$. As each vertex of $B$ is adjacent to at least three vertices of $A$, each vertex of $A$ can be in at most $2\dot 5+30-24=16$ pairs of vertices of $A$ which share a neighbor in $B$. 

If $A\simeq K_3\sqcup K_3$, then by Lemma \ref{i3}, each pair $\{a,b\}$ and $\{a,c\}$ must have at least 2 neighbors in $B$, since they each share exactly 1 neighbor in $A$. By Lemma \ref{j4}, each of the pairs$\{a,d\}$, $\{a,e\}$, and $\{a,f\}$ must share at least 4 neighbors in $B$, since they share no neighbors in $A$. As $A$ is vertex-transitive, it follows that each vertex of $A$ must be in at least $2+2+4+4+4=16$ pairs of $A$ with a shared vertex in $B$. This can only happen if exactly two vertices of $B$ are adjacent to all the vertices of $A$, end each of the remaining vertices of $B$ is adjacent to exactly 3 vertices of $A$.
Let $v$ be a vertex of $B$ which is adjacent to all vertices of $A$. If $a$ and $b$ share only two neighbors in $B$, then, as a set, $\{a,b\}$ is incident to every vertex of B, and thus $\{a,v,b\}$ is a connected dominating set and $G$ is 3-compliant. Else, $a$ and $b$ share at least three neighbors in $B$ and thus $a$ must be part of at least $3+2+4+4+4=17>16$ pairs of vertices of $A$ with a shared neighbor in B, a contradiction.

If $A\simeq C_6$, then by Lemma \ref{i3}, each pair $\{a,b\}$ and $\{a,f\}$ must have at least 3 neighbors in $B$, since share no neighbors in $A$. By Lemma \ref{j4}, each of the pairs$\{a,c\}$ and $\{a,e\}$ must share at least 3 neighbors in $B$, since they share exactly 1 neighbor in $A$. By Lemma \ref{j4}, $a$ and $d$ must have at least 4 common neighbors in $B$, as they have no shared vertices in $A$. Since $A$ is vertex-transitive, it follows that each vertex of $A$ must be in at least $3+3+3+3+4=16$ pairs of $A$ with a shared vertex in $B$. This can only happen if exactly two vertices of $B$ are adjacent to all the vertices of $A$, end each of the remaining vertices of $B$ is adjacent to exactly 3 vertices of $A$. Then $G$ has a pair of twin vertices $u$ and $v$ of degree 6, adjacent to all the vertices of $A$. By Proposition \ref{twin}, $G-v$ is a 3-non-compliant graph of order 13. But Theorem \ref{QR13} shows that $G-v\simeq QR_{13}$, thus $G$ is obtained by adding a twin vertex which connects to the open neighborhood of a vertex in $QR_{13}$. Theorem \ref{QR13} guarantees that $G$ is 3-noncompliant, and so is its complement, $\overline{G}$, which is the graph obtained by adding a twin vertex which connects to the closed neighborhood of a vertex in $QR_{13}$.

\end{proof}

Note that the only two non-compliant graphs of order 14 have the $QR_{13}$-graph as an induced subgraph. We also found fourteen 3-non-compliant graphs of order 15 (seven pairs) which all have a subgraph isomorphic to $QR_{13}$. We provide their edge lists in Appendix 1 \cite{LPP}. On the other hand, $QR_{17}$-graph is 3-non-compliant, yet it does not have any subgraph isomorphic to $QR_{13}$. However, $QR_{13}$ is a minor of $QR_{17}$. We would like to pose the following questions:\\

\noindent{\textbf{Question 2}: What is the complete list of 3-non-compliant graphs of order 15?\\
\noindent{\textbf{Question 3}: Is there an example of a 3-non-compliant graph which does not have $QR_{13}$ as a minor?

\section{Three-non-compliant graphs of order 15}

\begin{lemma}
    Let $G$ be a graph with $15$ vertices and $u\in V(G)$ with $\deg_G(u)=6$. If $G$ is 3-non-compliant, then
    $$\sum_{v\in N(u)}\deg_G{v}\ge 42.$$
    \label{lem:sumdeg42}
\end{lemma}
\begin{proof}
Let $N=N_G(u)$, $H=G-(N\cup u)$, and $L$ be the bipartite graph whose edges have one endpoint in $N$ and one endpoint in $H$.
By Lemma \ref{j3}, each vertex of $H$ shares at least three neighbors with $u$, thus $|L|\ge 24.$
By Lemma \ref{i2}, $u$ shares at least two neighbors with each of of its neighbors, thus 
$\deg_N({v})\ge 2,$ for each $v\in V(N)$ and $|N|\ge 6$.
We have $$\sum_{v\in V(N)}\deg_G(v) = 6 + 2|N| +|L|\ge 6 + 12 + 24 = 42. $$

\end{proof}

\begin{lemma} For $n=15$, graphs $G$ with $m=45, 46, 47$ or  $48$ edges are 3-compliant. 
\label{lem:edges15}
\end{lemma}

\begin{proof}
For $m=45, 46, 47$, consider a vertex $u\in V(G)$ with $\deg(u)=6$. The sum of degrees for the neighbors of $u$ is strictly less than 42. By Lemma \ref{lem:sumdeg42}, the graph $G$ is 3-compliant.
\\
For $m=48$, independent of its degree sequence, the graph $G$ has at least nine vertices of degree 6. 
There is at least one vertex $u$ with $\deg(u)=6$ such that $u$ does not neighbor all vertices of degree 7 or 8. Then the sum of degrees for the neighbors of $u$ is strictly less than 42. By Lemma \ref{lem:sumdeg42}, the graph $G$ is 3-compliant.
\end{proof}

\begin{lemma}  Let $G$ be a graph with $15$ vertices and  $v\in V(G)$ with $\deg_G(v)=6$. If 
    $$\sum_{u\in N(v)}\deg_G{u}\le 45,$$
    then at least one of the following three statements is true: $G$ is 3-compliant; $G$ has a $K_7$ minor; $\overline{G}$ has a $K_7$ minor.
\label{l45}
\end{lemma}
\begin{proof}

Let   $v\in V(G)$ with $\deg_G(v)=6.$
Let $N=N_G(v)$, $H=G-(N\cup v)$, and $L$ be the bipartite graph whose edges have one endpoint in $N$ and one endpoint in $H$.
We denote by $v_1, v_2, \ldots, v_8$ the vertices of $H$.

Assume $G$ is 3-non-compliant. 
By Lemmas \ref{j3} and \ref{i2},  $v$ shares three neighbors with each vertex of $H$ and two neighbors which each vertex of $N$.
This means  $deg_L(u)\ge 3$ for all $u\in H$, and $deg_N(u)\ge 2$ for each $u\in N$. This means $|L|\ge 24$ and $|N|\ge 6$, and 
since 
  $$ \sum_{u\in N(v)}\deg_G{u} = 6 + 2|N| + |L| \le 45,$$
$|N|=6$ or $|N|=7.$
With these restrictions, $N$ can be one of the seven graphs in Figure \ref{7Ns}.
The respective complements of these seven graphs can be seen in Figure \ref{7Nsc}.

\begin{figure}[htpb!]
\begin{center}
\begin{picture}(330, 150)
\put(0,0){\includegraphics[width=5in]{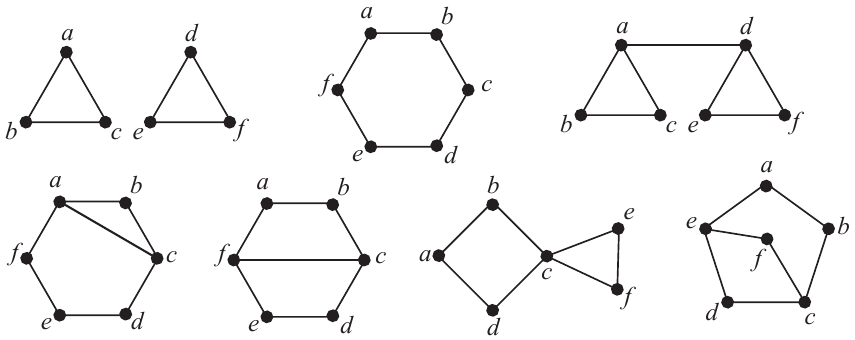}}
\end{picture}
\caption{\small Graphs of order 6,  size 6 or 7, and minimal degree 2.} 
\label{7Ns}
\end{center}
\end{figure}

\begin{figure}[htpb!]
\begin{center}
\begin{picture}(380, 140)
\put(0,0){\includegraphics[width=5.5in]{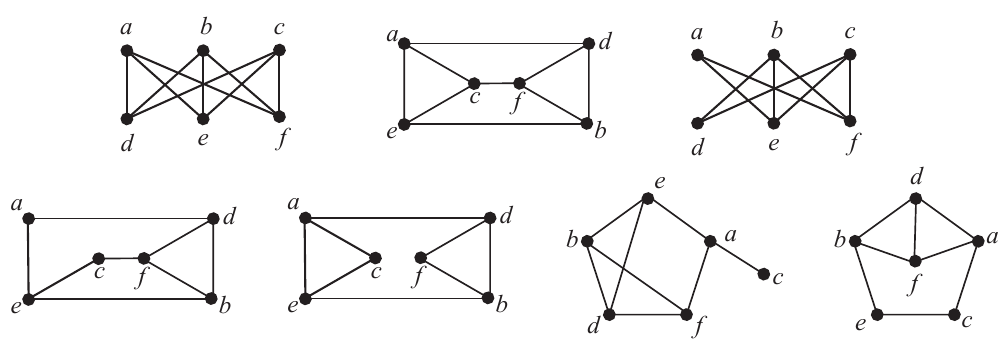}}
\end{picture}
\caption{\small Complements of the seven graphs of order 6,  size 6 or 7, and minimal degree 2.} 
\label{7Nsc}
\end{center}
\end{figure}

In each of seven cases, we show that under the assumption that $G$ is 3-non-compliant, $G$ or $\overline{G}$ has a $K_7$ minor. 
We primarily work in the complement $\overline{G}$. We note that regardless of $N$, in the complement $\overline{G}$, each pair of vertices of $N$ have a common neighbor in $\overline{H}$. 
Otherwise, this pair of vertices together with $v$ form a connected domination set of cardinality 3 in $G$, contradicting the assumption that $G$ is 3-non-compliant. 
We also note that each vertex of $\overline{H}$ neighbors at most three vertices of $cN$. In each case below we consider a graph $N$ as labeled in Figure \ref{7Ns}.\\
\textit{Case 1}. $N$ is a disjoint union of two triangles.
A vertex of $H$, say $v_1$, cannot neighbor all three vertices $a,b,c$ in $\overline{G}$. Otherwise $\{v, v_1, b\}$  dominates $\overline{G}$. 
Similarly, a vertex of $H$ cannot neighbor all vertices $d, e, f$. 
It follows that the edges $ab, bc, ac, de, df,$ and $ef$ are all covered (define covered)  by a different vertex of $\overline{H}$. 
Six edge contractions give a graph which contains a $K_7$ subgraph induced by $\{a, b, c, d, e, f, v\}$.\\
 \textit{Case 2}. $N$ is a 6-cycle.
  Each edge $ab, cd, $ and $ef $  is covered by a different vertex of $\overline{H}$. 
 Also, each edge $bc, de,$ and $af$ is covered by a different vertex of $\overline{H}$. 
 If a single vertex of $\overline{H}$, say $v_1$ covers two edges, say $ab$ and $bc$, then $\{v, v_1, b\}$  dominates $\overline{G}$, a contradiction.
 It follows that the edges $ab, cd, ef, bc, de,$ and $af$ are all covered  by a different vertex of $\overline{H}$. 
Six edge contractions give a graph which contains a $K_7$ subgraph induced by $\{a, b, c, d, e, f, v\}$.\\
\textit{Case 3}. $N$ is a union of two triangles and an edge.
As in Case 1, the edges $ab, bc, ac, de, df,$ and $ef$ are all covered by a different vertex of $\overline{H}$. 
Six edge contractions give a graph which contains a $K_7$ subgraph induced by $\{a, b, c, d, e, f, v\}$.\\
If the edge $cf$ is covered by same vertex that covers $ac$ or $bc$,  say $v_2$, then edge $cv_2$ is contracted.
If the edge $cf$ is covered by same vertex that covers $df$ or $ef$,  say $v_3$, then edge $fv_3$ is contracted.\\
\textit{Case 4}. $N$ is a 6-cycle together with a chord of length 2.
Since $\{d, e, f\}$ does not  dominate $G$, the vertices $d, e, $ and $f$ have a common neighbor in $\overline{H}$, $v_1$. 
Then $\{v, v_1, e\}$ dominate $\overline{G}$, a contradiction.\\
\textit{Case 5}. $N$ is a 6-cycle together with a chord of length 3.
Since $\{a, e, f\}$ does not dominate $G$, the vertices $a, e, $ and $f$ have a common neighbor in $\overline{H}$, $v_1$. 
Since $\{b, c, d\}$ does not  dominate $G$, the vertices $a, e, $ and $f$ have a common neighbor in $\overline{H}$, $v_2$. 
Then $\{v, v_1, v_2\}$ dominate $\overline{G}$, a contradiction.\\
\textit{Case 6}. $N$ is a 4-cycle and a 3-cycle that share a vertex.
Since $\{a, c, d\}$ does not dominate $G$, the vertices $a, c, $ and $d$ have a common neighbor in $\overline{H}$, $v_1$. 
Then $\{v, v_1, d\}$ dominate $\overline{G}$, a contradiction.\\
\textit{Case 7}. $N$ is a 5-cycle together with a vertex that neighbors two non-adjacent vertices of the 5-cycle.
Since $\{a, e, f\}$ does not dominate $G$, the vertices $a, e, $ and $f$ have a common neighbor in $\overline{H}$, $v_1$. 
Since $\{b, c, d\}$ does not  dominate $G$, the vertices $a, e, $ and $f$ have a common neighbor in $\overline{H}$, $v_2$. 
Then $\{v, v_1, v_2\}$ dominate $\overline{G}$, a contradiction.

\end{proof}

\begin{lemma}  Let $G$ be a graph with $15$ vertices and $m=49, 50, 51,$ or $52$ edges. Then there exists $H\in \{G, \overline{G}\}$ and $v\in V(H)$ such that 
   $$\sum_{u\in N_H(v)}\deg_H{u}\le 45.$$
\label{l46}
\end{lemma}

\begin{proof}
Let $a$, $b$ and $c$ represent the number of vertices of degree 6, 7, and 8, respectively in $G$.
Then  $a$, $b$ and $c$ represent the number of vertices of degree 8, 7, and 6, respectively in $\overline{G}$.
Assume by contradiction that for both $G$ and $\overline{G}$, for all vertices of degree 6, the sum of the degrees of their neighbors is at least 46.
Then each vertex of degree 6 must neighbor at least four vertices of degree 8. 
This gives  $4a+4c\le ac,$ where $ac$ represents the number of edges of a complete bipartite graph $K_{a,c}.$
We also have $a+b+c= 15$ and $6a+7b+8c= 2m$. The last two equalities give $c=105-2m+a.$ We let $k=2m-97\in \{1, 3,5,7\}$ and substitute $c=a-k+8$  in the inequality $4a+4c\le ac,$   to obtain 
$$a^2-ka+4k-32\ge 0.$$
 This inequality implies that 
 $$a \ge \frac{k+\sqrt{(k-8)^2+64}}{2}.$$
  For $m=49,50,51$, and 52, the inequality implies $a \ge 6, 6, 7$ and $8$, respectively. 
 On the other hand,  for $m=49, 50, 51$, and 52, the graph $G$ can have at most 4, 5, 6, and 7 vertices of degree 8, respectively. This is when $G$ has no vertices of degree 7.
 We get a contradiction in each case.
\end{proof}

\begin{theorem}
If $G$ is  graph with 15 vertices,   then at least one of the following three statements is true: $G$ is 3-compliant; $G$ has a $K_7$ minor; $\overline{G}$ has a $K_7$ minor.
\label{oneofthree}
\end{theorem}


\section{$K_{15}$ is not bi-nIK}

The properties of being nIL or nIK are \textit{hereditary}, which means that any minor of a graph which is nIL (nIK), is itself nIL (nIK). Considering complements, any graph which has an IL minor is itself IL, and any graph which has and IK minor is itself IK. Conway and Gordon \cite{CG} proved that $K_6$ is IL, and that $K_7$ is IK. The combined work of Conway and Gordon \cite{CG}, Sachs \cite{Sa}, and Robertson, Seymour and Thomas \cite{RST} fully characterize IL graphs: a graph is IL if and only if it contains a graph in the Petersen family as a minor.
The Petersen family consists of seven graphs obtained from $K_6$ by $\nabla Y-$moves and $Y\nabla-$moves, as presented in Figure \ref{fig-ty}.

\begin{figure}[htpb!]
\begin{center}
\begin{picture}(160, 50)
\put(0,0){\includegraphics[width=2.4in]{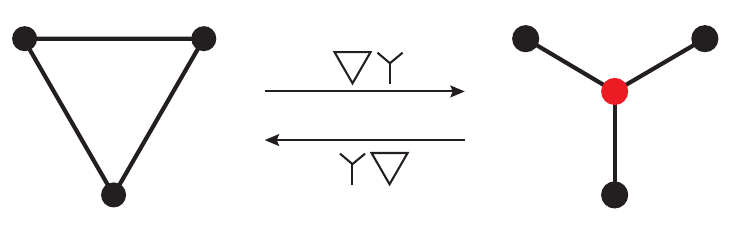}}
\end{picture}
\caption{\small  $\nabla Y-$ and $Y\nabla-$moves} 
\label{fig-ty}
\end{center}
\end{figure}

In the absence of a known IK minor, it is relatively difficult to prove a given graph is IK. However, under certain linking conditions on the graph, it follows that the graph is IK. The $D_4$ graph is the (multi)graph shown in Figure~\ref{fig-D4}.
A \dfn{double-linked $D_4$} is a $D_4$ graph embedded in $S^3$
such that each pair of opposite 2-cycles ($C_1 \cup C_3$, and $C_2 \cup C_4$) has odd linking number.
The following lemma was proved by Foisy \cite{Fo2};
a more general version of it was proved independently 
by Taniyama and Yasuhara \cite{TY}.

\begin{figure}[htpb!]
\begin{center}
\begin{picture}(150, 130)
\put(0,0){\includegraphics[width=1.8in]{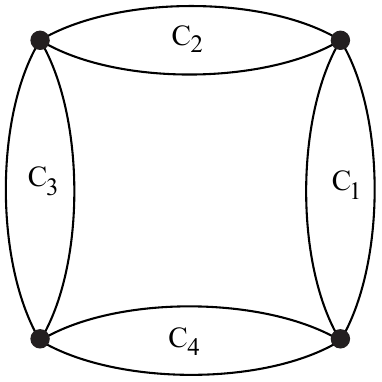}}
\end{picture}
\caption{The $D_4$ graph.} 
\label{fig-D4}
\end{center}
\end{figure}

\begin{lemma}
\label{D4-Lemma}
Every double-linked $D_4$ contains a nontrivial knot.
\end{lemma}

A graph which contains a double-linked D4-minor in every embedding is called \textit{intrinsically D4} (\dfn{ID4}). Lemma \ref{D4-Lemma}
remains the gold standard in showing a graph is IK, since any intrinsically D4 graph is intrinsically linked. For this reason, Miller and Naimi \cite{MN} developed 
an algorithm, implemented as a Mathematica program, to determine whether a graph is intrinsically D4 .

We used their program and other Mathematica code developed with Ramin Naimi \cite{N} to prove the following theorem.
\begin{theorem} Let $G$ be a nIL graph of order $12$. Then $\overline{G}$ is intrinsically knotted (IK).
\label{12IK}
\end{theorem}
\begin{proof}
There are a total of 6503 maximal linklessly embeddable (\dfn{maxnIL}) graphs of order 12 \cite{P}. The search was organized following the following observations:

\begin{itemize}
\item If $\delta(G)$ is the minimal degree of a maxnIL graph $G$, then $2\le \delta(G)$. It easy straightforward to check that any nIL graph with vertices of degree 0 or 1 is a proper subgraph of another nIL graph of the same order.
\item If $G$ is a maxnIL graph of order 12, then $\delta(G)\le 5$.  This follows from \cite{OP}, where the second author and his student proved that every graph of order 12 with minimal degree at least 6 has a $K_6$-minor, and it is therefore intrinsically linked.
\item When deleting a vertex from a maxnIL graph of order 12, one obtains a nIL graph of order 11. 
\end{itemize}

The previous observations show that any maxnIL graph of order 12 is an edge-deletion subgraph of a graph obtained by adding a vertex of degree at most 5 to a maxniIL graph of order 11. By the work of the last two authors, Ryan Odeneal, and Ramin Naimi \cite{NOPP}, any maxnIL graph which has a vertex of degree 2 or 3, can be obtained from a maxnIL of order one less by performing a clique sum with a $K_3$ over a $K_2$ induced subgraph, or a clique sum with a $K_4$ over an induced $K_3$-subgraph. We took the list of 710 maxnIL graphs of order 11, obtained the graphs of order 12 through the clique sum operation, and sifted out the 2273 maxnIL graphs with minimal degree at most 3. The remaining 2230 were found by adding a vertex of degree 4 or 5, respectively, checked whether any nIL graphs were obtained, and which one were maximal. Then we took the ones that were intrinsically linked, used McKay and Piperno's nauty program \cite{MPip} to obtain all their 1-edge deletion subgraphs up to isomorphism, selected those which has minimum degree at least 4, and then sorted them into IL or nIL. We repeated the process until no graphs with minimal degree 4, respectively 5 were left. 

We then used Miller and Naimi's program \cite{MN} and verified that all the complements of the 6503 maxnIL graphs of order 12 were intrinsically D4. Since any nIL graph of order 12 can be completed, by adding edges, to a maxnIL of order 12, it follows that the complement of a nIL graph of order 12 contains the complement of a maxnIL graph of order 12, and it is therefore intrinsically knotted.
\end{proof}

\begin{observation} Let $G$ be an IL graph. Then $G \ast K_1$ is IK.
\label{coneoverIL}
\end{observation}
\begin{proof}
Since the statement holds for all the members of the Petersen family, the result follows.
\end{proof}

\begin{theorem}
Let $G$ be a simple graph on 15 vertices and let $\overline{G}$ denote its complement. 
If  $G$  is nIK, then $\overline{G}$ is IK.
\label{15binik}
\end{theorem}

\begin{proof}
If $G$ is 3-compliant, then either $G$ or $\overline{G}$ has a minor with a vertex of degree 12.  We assume this minor has order 13.
If  $G$ has a minor $H$ with a vertex $v$ of degree 12, since $G$ is nIK, by Theorem \ref{coneoverIL}, $H-v$ is nIL. 
By Theorem \ref{12IK}, $c(H-v)$ is IK, thus $\overline{G}$ is IK.
If $\overline{G}$ has a minor $H$ with a vertex $v$ of degree 12 and $\overline{G}$ is nIK,  $H-v$ is nIL. 
By Theorem \ref{12IK}, $c(H-v)$ is IK, thus $G$ is IK.

If $G$ is 3-non-compliant, by Theorem \ref{oneofthree}, at least one of $G$ or $\overline{G}$ has a $K_7$ minor. Since $K_7$ is IK, 
one of $G$ or $\overline{G}$ is intrinsically knotted. 
\end{proof}
Note that the previous result proves that $K_n$ is not bi-nIK, for all $n\ge 15$. On the other hand $K_{12}$ is bi-nIK, as the following example will demonstrate. The graph depicted on the left of Figure \ref{sc12} is self-complementary (isomorphic to its complement). It is also 2-apex, as the deletion of the "empty" vertices yields the planar graph on the right. As 2-apex graphs are knotlessly embeddable by results of Blain et al. \cite{BBFHL} and Ozawa and Tsutsumi, \cite{OT}, the graph in Figure \ref{sc12} is nIK. It follows $K_{12}$ is bi-nIK.

\begin{figure}[htpb!]
\begin{center}
\begin{picture}(230, 110)
\put(0,0){\includegraphics[width=3.5in]{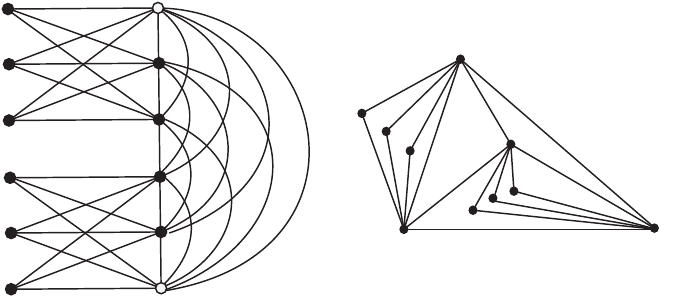}}
\end{picture}
\caption{A self-complementary, 2-apex graph of order 12.} 
\label{sc12}
\end{center}
\end{figure}

The following is a natural question:\\

\noindent{\textbf{Question 4}: Are either of $K_{13}$ or $K_{14}$ bi-nIK?\\

Note that in the arguments leading to the proof of Theorem \ref{oneofthree} and Theorem \ref{15binik}, we focused mainly on finding a $K_7$ minor in either the graph or its complement. Since not having a $K_7$($K_6$) minor is also a hereditary property, we pose the following questions:\\

\noindent{\textbf{Question 5}: What is the smallest integer $15\ge n>12$, such that $K_n$ is not bi-$K_6$-free?\\

\noindent{\textbf{Question 6}: What is the smallest integer $18 \ge n>13$, such that $K_n$ is not bi-$K_7$-free?\\

The upper bounds in the previous two questions are derived from a theorem of Mader \cite{Ma}, which implies that any graph of order $n$ with at least $4n-9$ edges has a $K_6$ minor, and any graph of order $n$ with at least $5n-14$ edges has a $K_7$ minor. The lower bounds stem from the example in Figure \ref{12k6free}. This order 12 graph has no $K_6$ minor, and neither does its complement.

\begin{figure}[htpb!]
\begin{center}
\begin{picture}(180, 125)
\put(0,0){\includegraphics[width=2.5in]{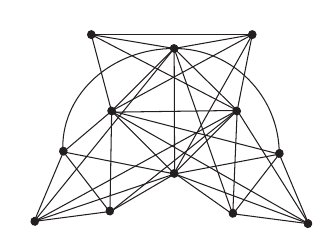}}
\end{picture}
\caption{A $K_6$-free graph, with its complement also $K_6$-free.} 
\label{12k6free}
\end{center}
\end{figure}


\section{k-Non-Compliant Graphs}

\begin{theorem}
\label{theo:compliance}
A graph $G$ or its complement $\overline{G}$ has a minor $H$ where $\Delta(H) \geq |G| - k = n-k$ if and only if $\min(\gamma_c(G),\gamma_c(\overline{G})) \leq k$.
\end{theorem}

\begin{proof}
The statement follows from the definitions if $k=1$, or $G$ is disconnected.
For the direct implication, without loss of generality, let $H$ be a minor of $G$ and $s \in V(H)$ be such that $\deg_H(s) = n-k.$ Let $S \subset V(G)$ be exactly the set of vertices contracted through edges in $G$ to produce $s \in V(H)$. 
We note that $G[S]$ is connected. 
Define $N_S = \{u \in V(G)\setminus S\:|\: \exists \: v \in S: (u,v) \in E(G)\}$.
By assumption, $|N_S| \geq n-k$. If $N_S = V(G)\setminus S,$ then $S$ is a connected dominating set of size at most $k,$ implying $\gamma_c(G) \leq k.$

If $N_S \subsetneq V(G)\setminus S,$ then $T = V(G)\setminus (S \sqcup N_S)$ is non-empty such that (in $G$) no vertex in $S$ is adjacent to a vertex in $T$. 
This gives  $\overline{G}[S\sqcup T]$ is connected. Moreover, if every vertex $u \in N_S = V(G) \setminus (S \sqcup T)$ is adjacent in the complement $\overline{G}$ to some $v \in S \sqcup T$, then $S \sqcup T$ is a connected dominating set and $\gamma_c(\overline{G}) \leq k.$

So, let there exist some $u_0 \in N_S$ such that $(u_0, v) \in E(G)$ for every $v \in S\sqcup T.$ 
Let $S': = S \sqcup \{u_0\}$.
Note that $G[S']$ is connected, and since $T$ was nonempty, we have $|S'| \leq k$. 
Because $u_0 \in S'$, we find $S'$ is a connected dominating set of $G$. We conclude, in this third case $\gamma_c(G) \leq k$.

Towards the reverse implication, without loss of generality, let $\gamma_c(G) \leq k$. Let $S \subset V(G)$ be a connected dominating set of $G$ such that $|S| \leq k$. 
A minor $H$ with $\Delta(H)\ge n-k$  is obtained from $G$ by contracting the edges of any spanning tree of the connected subgraph $G[S]$.
\end{proof}

Define $f: \N \to \N$ so that $f(n)$ is the largest number such that for all graphs $G$ of order $n$,  $G$ or $\overline{G}$ has a minor $H$ with $\Delta(H) \geq f(n)$.
The first few values  of $f(n)$ are $f(1)=1, f(2)=2, f(3)=2, f(4)=2, f(5)=3$. Note that $n - f(n)$ corresponds to the minimum value of $k$ such that all graphs $G$ of order $n$ are $k$-compliant. The construction in Proposition \ref{twin} shows that $\{n-f(n)\}$ is a non-decreasing sequence.

\begin{theorem}
\label{theo:linearbound}
For $n \geq 15,$
$$n - \left\lfloor \frac{n+1}{4}\right\rfloor \leq f(n).$$
\end{theorem}

\begin{proof}
If $G$ is 3-compliant, the inequality holds since $$n-\left\lfloor \frac{n+1}{4} \right\rfloor \le n-3\le f(n), \textrm{for all } n\ge 15.$$
Assume $G$ is 3-non-compliant.
 Then $\gamma_c(G) + \gamma_c(\overline{G}) \leq \delta^*(G) + 2$ by Theorem 5 in \cite{KSKW}.
If $\delta^*(G) \not= 6$, by Theorem 7 in \cite{KSKW}, $\gamma_c(G) + \gamma_c(\overline{G}) < \delta^*(G) + 2$.
Since  $\delta^*(G) \leq \lfloor \frac{n-1}{2} \rfloor$, then
$$\gamma_c(G) + \gamma_c(\overline{G}) < \left\lfloor \frac{n-1}{2} \right\rfloor + 2,$$ or equivalently $$ \gamma_c(G) + \gamma_c(\overline{G}) \leq \left\lfloor \frac{n-1}{2} \right\rfloor + 1.$$
It follows that  \begin{equation}
\label{eq:linfn}
    \min(\gamma_c(G), \gamma_c(\overline{G})) \leq \left\lfloor  \frac{\lfloor \frac{n-1}{2} \rfloor + 1}{2}\right \rfloor = \left\lfloor \frac{n+1}{4} \right\rfloor.
\end{equation}
By Theorem \ref{theo:compliance},
$$n - \left\lfloor \frac{n+1}{4}\right\rfloor \leq f(n).$$
If $\delta^*(G) = 6$, since $G$ is 3-non-compliant,  then $\gamma_c(G) = \gamma_c(\overline{G}) = 4$. 
By Theorem \ref{theo:compliance}, $f(n)\ge n-4$. 
For $n \geq 15$, $\left\lfloor \frac{n+1}{4} \right\rfloor \geq 4$ and it follows that $f(n) \geq n - \left\lfloor \frac{n+1}{4} \right\rfloor$.

\end{proof}

\begin{remark} Note that $n - \left\lfloor \frac{n+1}{4} \right\rfloor = n-3$, for $n=13, 14$.
So the bound $n\ge 15$ is necessary, by Theorems \ref{QR13} and \ref{order14}.
\end{remark}

Note that Theorem \ref{theo:linearbound} shows that all the graph of order $n\le 18$ are 4-compliant. On the other hand, $QR_{61}$ is 4-non-compliant. This begs the questions:\\

\noindent{\textbf{Question 5}: What is the smallest order of a 4-non-compliant graph?\\

\noindent{\textbf{Question 6}: What is the smallest order of a $k$-non-compliant graph, for $k\ge 4$?\\

\begin{theorem}[Karami et al. \cite{KSKW}]
\label{theo:productbound}
If $G$ and $\overline{G}$ are both connected of order $n \geq 7$, then 
$$\gamma_c(G) \cdot \gamma_c(\overline{G}) \leq 2n - 4,$$
with equality if and only if $G$ or $\overline{G}$ is a path or a cycle.
\end{theorem}

\begin{corollary}
For $n \geq 7$, we have
$$\min\{\gamma_c(G), \gamma_c(\overline{G})\} \leq \left\lceil \sqrt{2n - 4}\right\rceil - 1, \textrm{ and equivalently } n - \left\lceil \sqrt{2n - 4}\right\rceil + 1 \leq f(n).$$
\end{corollary}

\begin{proof}
If $n\ge 7$ and $G$ is a path or a cycle, then $\gamma_c(G)=n-2$ and $\gamma_c(\overline{G})=2$. 
Then $\min\{\gamma_c(G), \gamma_c(\overline{G})\} =2  \leq \left\lceil \sqrt{2n - 4}\right\rceil - 1.$
For $G$ not a path or a cycle, by Theorem \ref{theo:productbound}, $\gamma_c(G) \cdot \gamma_c(\overline{G}) < 2n - 4.$
Assume $\gamma_c(G), \gamma_c(\overline{G}) >\left\lceil \sqrt{2n - 4}\right\rceil - 1$.
Then $\gamma_c(G), \gamma_c(\overline{G}) \ge \left\lceil \sqrt{2n - 4}\right\rceil $, and $\gamma_c(G)\cdot \gamma_c(\overline{G}) \ge  \big(\left\lceil \sqrt{2n - 4}\right\rceil \big)^2 \ge 2n-4$, a contradiction.
\end{proof}

\begin{remark}
This bound is a strict improvement over Theorem \ref{theo:linearbound} starting from $n \geq 31$.
\end{remark}

\end{document}